\documentclass{ipart}

\startlocaldefs
\usepackage{abstract}
\usepackage[utf8]{inputenc}
\usepackage{mathtools}
\usepackage{graphicx}
\usepackage{color}
\usepackage{ verbatim}
\usepackage{amsmath}
\usepackage{amsfonts}
\usepackage{amsthm}
\newtheorem{theorem}{Theorem}
\newtheorem{prop}{Proposition}
\newtheorem{cor}{Corollary}
\newtheorem{lemma}{Lemma}
\theoremstyle{definition}
\newtheorem{exmp}{Example}
\newtheorem*{obs}{Reformulation of Stanley's Conjecture}

\endlocaldefs

\begin{document}

\begin{frontmatter}

\title{Transplanting Trees: Chromatic Symmetric Function Results through the Group Algebra of $S_n$}
\runtitle{Transplanting Trees}


\author{\fnms{Ang\`ele M.} \snm{Foley}},
\address{Department of Physics and Computer Science\\
Wilfrid Laurier University\\
 Waterloo, Ontario, Canada}
\author{\fnms{Joshua} \snm{Kazdan},\ead[label=e1]{}}
\address{University of California Los Angeles \\
Los Angeles, California, USA}
\author{\fnms{Larissa} \snm{Kr\"{o}ll}},
\address{University of Waterloo\\
Waterloo, Ontario, Canada}
\author{\fnms{Sof\'{i}a} \snm{Mart\'{i}nez Alberga}},
\address{Purdue University\\
West Lafayette, Indiana, USA\space\space\space\space\space\space\space\space\space}
\author{\fnms{Oleksii} \snm{Melnyk}},
\address{University of Oxford\\
Oxford, UK}
\and
\author{\fnms{Alexander} \snm{Tenenbaum}}
\address{University of California Los Angeles\\
Los Angeles, California, USA}

\runauthor{Foley, Kazdan, Kr\"oll, Mart\'{i}nez Alberga, Melnyk, Tenenbaum}

\end{frontmatter}

\maketitle

\begin{abstract}
One of the major outstanding conjectures in the study of chromatic symmetric functions (CSF's) states that trees are uniquely determined by their CSF's.  Though verified on graphs of order up to twenty-nine, this result has been proved only for certain subclasses of trees.  Using the definition of the CSF that emerges via the Frobenius character map applied to $\mathbb{C}[S_n]$, we offer new algebraic proofs of several results about the CSF's of trees.  Additionally, we prove that a ``parent function" of the CSF defined in the group ring of $S_n$ can uniquely determine trees, providing further support for Stanley's conjecture.
\end{abstract}

\section{Introduction}\label{intro}

Stanley's 1994 paper \cite{Stan} introduced the chromatic symmetric function (CSF) and conjectured that non-isomorphic trees have distinct CSF's.   For the general question---can two different graphs have the same chromatic symmetric function?---the answer is ``yes,'' as evidenced by Figure 1 of \cite{Stan}. However, Stanley asked, can two different trees on the same number of vertices have the same chromatic symmetric function?  Stanley conjectured that the chromatic symmetric function (CSF) can distinguish trees.

The ability of the CSF to distinguish trees has been the subject of many papers, including \cite{Foug},\cite{MarMorWag},\cite{OrScott}, \cite{Tsuj}.  Partial proofs have demonstrated that caterpillars and spiders \cite{MarMorWag} exhibit unique chromatic symmetric functions. The conjecture has been computationally verified by Russell \cite{Russ} on trees with up to 25 vertices, and more recently by Heil and Ji \cite{HeiJi} on trees with up to 29 vertices.

Inspired by a group algebra approach first proposed by Stanley \cite{Enum} and substantially developed in recent years by Pawlowski  \cite{Paw},
we supply new proofs for a number of results.  The group algebra approach allows for  cohesive proofs of basic results concerning CSF's.  We obtain the CSF of a graph without cycles $G=(V,E)$ by applying the Frobenius character map to the function
\begin{equation}
    \mathcal{K}_{\mathcal{L},\pi}(G)= n! \prod_{(i,j)\in E} (1-(ij))
    \end{equation}
where $\mathcal{L}$ is a labeling of the vertices of $G$ and $\pi$ is an ordering of the edges of $G$.
Our main theorem demonstrates that Stanley's conjecture on trees can be reformulated as a question about equality of these $\mathcal{K}_\pi$ functions under conjugation. For this see 

\section{Background and Notation}

We begin by establishing basic notation, and refer the reader to West \cite{West} or Sagan \cite{Sag} for additional details on graphs or symmetric functions. Let $G = (V,E)$ be a graph where $V$ and $E$ denote the vertex and edge set respectively. For clarity, when $i,j \in V$ and $(i,j) \in E$,  we will call $i$ and $j$ $\emph{endpoints}$. If $V' \subset V $ and $E' \subset E $, where members of $E'$ have endpoints in $V'$,  then $H= (V', E') $ is called a $\emph{subgraph}$ of $G$. A graph $T=(V,E)$ is called a $\emph{tree}$ if it is connected and does not contain a cycle. If $|V|=n$, we will call $T$ a \emph{tree of order} $n$. A graph $F$ is called a $\emph{forest}$ if it is a disjoint union of trees. A vertex $v \in T $ is a $\emph{leaf}$ if the degree of $v$ (number of adjacent edges) is $1$.

There are several polynomials which can be identified with a graph. Of particular interest to us is the matching polynomial of a graph, $G=(V,E)$. To define the matching polynomial, we first need to define a $k$-matching on $G$, which is a subset $E'\subset E$ with $|E'| = k$ such that no two members of $E'$ share a vertex. The matching polynomial $\mu$(x) of $G$ is defined as 
\begin{equation}
\mu(x)\coloneqq \sum_{k=0}^{|V|}m_kx^k,
\end{equation}
where $m_k$ denotes the number of $k$-matchings.

A $\emph{coloring}$ of a graph $G=(V,E)$ is an assignment $ \kappa : V \rightarrow \mathbb{N} $.  We say that a coloring $\kappa$ is $\emph{proper}$ if  $\kappa(i) \neq \kappa(j)$, for all $(i,j) \in E$.

If $G$ has a vertex set $V= {v_1, v_2, v_3, ...., v_n}$, then the $\emph{chromatic symmetric function}$ \cite{Stan} of $G$  is defined as
\begin{equation}\label{CSF}
X_{G}\coloneqq \sum_{\kappa}  x_{\kappa(v_1)}  x_{\kappa(v_2)}  x_{\kappa(v_3)} ...   x_{\kappa(v_n)}
\end{equation}
where the sum is over all proper colorings $\kappa$ of $G$.

In (\ref{CSF}), the CSF is a polynomial written in the monomial basis for symmetric functions. A monomial symmetric function is defined by fixing a $\emph{partition}$ $\lambda$ of $n$, i.e. $\lambda = (\lambda_1, \lambda_2, \lambda_3, ..., \lambda_{k})\in \mathbb{N}^k$, such that $\lambda_1 \geq \lambda_2 \geq \lambda_3 \geq ... \geq \lambda_{k}$. Then the corresponding monomial symmetric function $m_\lambda$ is the sum over all distinct monomials of the form

\begin{equation}
m_\lambda(x)=\sum_\alpha
x_{1}^{\alpha_1}  x_{2}^{\alpha_2}  x_{3}^{\alpha_3} ... x_{n}^{\alpha_{k}}.
\end{equation}
where the sum is over all distinct permutations of $\alpha=(\alpha_1, \ldots, \alpha_k)$ of the entries of $\lambda=(\lambda_1, \lambda_2, \ldots, \lambda_k)$.

Note that these monomials form a basis for the ring of symmetric functions $\Lambda^n$. Another basis consists of the $\emph{power sum symmetric functions}$, which are of particular importance in this paper. The $n$th power sum symmetric function is defined as 

\begin{equation}
p_n \coloneqq \sum_{i \geq 1} x_i^n.
\end{equation}
For a partition $\lambda = (\lambda_1,...,\lambda_k)$ of $n$ we define
$p_\lambda = p_{\lambda_1}\cdot \ldots \cdot p_{\lambda_k}.$

These symmetric functions form a basis of the space $\Lambda^n$. For full details see Macdonald \cite{MacDon}, Sagan \cite{Sag} and Stanley \cite{Enum}.

An alternative definition of the chromatic symmetric function employs the \emph{symmetric group ring} $\mathbb{C}[S_n]$.  The algebra $\mathbb{C}[S_n]$ consists of all formal finite sums \[\sum z_i \sigma_i\] where the sum is over $z_i \in \mathbb{C}$ and $\sigma_i\in S_n$.  The sum and multiplication of two elements in $\mathbb{C}[S_n]$ are defined as follows
\begin{equation}
\left(\sum_{i \in I} c_i \sigma_i\right) + \left(\sum_{i \in I} z_i \sigma_i\right) =  \sum_{i \in I} (c_i + z_i) \sigma_i 
\end{equation}
\begin{equation}
\left(\sum_{i \in I} c_i \sigma_i\right) \left(\sum_{j \in J} z_j \sigma_j\right) = \sum_{i\in I, j\in J} (c_i  z_j) \sigma_i\sigma_j.
\end{equation}\\
The \emph{Frobenius character map} is defined by:

\begin{equation}
\textrm{ch}: \mathbb{C}[S_n] \rightarrow \Lambda.
\end{equation}
\begin{equation}\label{frob}
\sigma \mapsto\frac{1}{n!} p_{\textrm{cyc}(\sigma)},
\end{equation}
where cyc denotes the cycle type.

\begin{exmp}
The Frobenius character map takes $(12)(3)(4)$ to $\frac{1}{4!}p_{2,1,1}$.
\end{exmp}

We can construct the CSF of a graph $G= (V, E)$ with $|V|=n$ as follows.  Assign labels in $\{1,...,n\}$ to the vertices in $V$.  One can treat the edges $(i, j) \in E$ as transpositions of the form $(ij)$ in $S_n$.  For a labeling $\mathcal{L}$ of $G$ and an ordering $\pi$ of the edges of $G$, define
\begin{equation}\mathcal{K}_{\mathcal{L},\pi}(G)\coloneqq n! \prod_{(i,j)\in E} (1-(ij)).\end{equation} This is similar to the function called $\alpha_F$ in \cite{Paw}, except that we specify an ordering $\pi$ of the edges and a labeling $\mathcal{L}$ of the vertices, whereas Pawlowski does not. Note that the ordering is significant because $\mathbf{C}[S_n]$ is a non-commutative ring. For convenience, we will generally omit $\mathcal{L}$ and write $\mathcal{K}_{\pi}$ instead of $\mathcal{K}_{\mathcal{L},\pi}$ except when the labeling is relevant to our proofs.  When one applies the Frobenius character map ch, one recovers $\textrm{ch}(\mathcal{K}_\pi(G)) = X_G$, the CSF of $G$.  

Remarkably, any vertex labelling and ordering $\pi$ of the product yields the same CSF when the Frobenius map is applied to a tree by by \cite{Paw} Theorem 2.5.

\begin{exmp}
Consider the tree $T= (V, E)$ with vertex set $V= \{1,2,3, 4\}$ and edge set $\{(1,2), (1,3), (1,4)\}$.  For the ordering $\pi = ((1,2), (1,3), (1,4))$, one finds that \begin{align}\nonumber \mathcal{K}_{\pi}(T) &= 24(1-(12))(1-(13))(1-(14)) \\ \nonumber &= 24 (1)(2)(3)(4)-24 (13)(2)(4) - 24(12)(3)(4) + 24(132)(4) \\&- 24(14)(3)(2) + 24(143)(2) + 24(142)(3) - 24(1342).\end{align}  Applying ch, one obtains \[X_G=p_{(1,1,1,1)}+ 3p_{(3,1)} - 3p_{(2,1,1)} - p_{(4)}.\]  One can convert this into the monomial basis in order to get \[m_{(3,1)}+ m_{(2,2,1)} + m_{(1,1,1,1)}.\]
\end{exmp}

\section{Trees Through the $S_n$ Group Algebra}
In the following section, we re-derive some important facts about the CSF's of trees by harnessing the group algebra construction of the CSF.  Although the results in this section are known, they have typically been derived using ad-hoc combinatorial techniques.  The algebraic proofs given in this section present a more cohesive approach to proving facts about CSF's.
For this section, when talking about trees we will be considering trees of order $n$ if not otherwise stated. Furthermore we will neither specify $\pi$ nor the assignment of integers to the vertices in the functions $\mathcal{K}(G)$, since it will not affect our proofs.

The following two results were originally proved in equation (3.5), \cite{Morin} through entirely different methods involving degree sequences.

\begin{theorem} \cite{Morin}
The matching polynomial can be recovered from the chromatic symmetric function.
\end{theorem}
\begin{proof}
Consider $\mathcal{K}(G)$ for a graph $G = (V,E)$ without cycles.  Let $S\subset E$ be a set of cardinality $k\leq n$ containing edges that pairwise share no endpoints.  In other words, $S$ is an independent set of edges.  Viewing $S \subset S_n$ as a set of transpositions, the product
\begin{equation}
\prod_{(ij) \in S} (ij).
\end{equation}
has order $2$.  This follows because cycles containing no shared edges commute and do not interact.  Thus,
\begin{equation}\label{prod_ind}
\mathrm{ch} \left(n! \prod_{(ij) \in S} (ij) \right) =  p_{(2^k,1^{n-k}) }
\end{equation}
Since every subset of $V$ comprised of $k$ independent edges contributes a product of the form \eqref{prod_ind} to the CSF of $G$, the coefficient of $p_{(2^k,1^{n-k})}$, which we denote  $c_{(2^k,1^{n-k})}$, counts the number of independent sets of cardinality $k$. It follows that the matching polynomial $\mu(x)$ can be recovered by

\begin{equation}
\mu(x)=\sum_{k} |c_{(2^k,1^{n-k})}|x^k.
\end{equation}

As the sign of the coefficient depends on the parity of $k$, i.e. $\mathrm{sgn}(c_{(2^k,1^{n-k})})=(-1)^k$, it is important to include the absolute value.
\end{proof}

\begin{theorem} \cite{Morin} \label{subtree}
The number of connected sub-trees of a tree $T$ with order $k$ can be obtained through its chromatic symmetric function.
\end{theorem}

\begin{proof}
Let $T=(V,E)$ be a tree with  
\begin{equation}
\mathcal{K}(T) = n! \prod_{(ij)\in E} (1-(ij)).
\end{equation}
For any subset $E'\subset E$, $\mathcal{K}(T)$ contains the product  
\begin{equation}\label{connectedprod}
\prod_{(ij)\in E'} (ij).
\end{equation}

Because no edges appear more than once as transpositions in \eqref{connectedprod}, the product viewed in $S_n$ forms a single cycle of length $|E'|+1$ if and only if the edges in $E$ form a connected subgraph of $T$.

All cycles of a given length have the same sign in $\mathcal{K}(G)$. Hence the number of connected subgraphs of $T$ with size $k$ is the absolute value of the coefficient of $p_{(k, 1^{n-k})}$ in the CSF of $T$.
\end{proof}

\begin{cor} \cite{Morin} \label{leafy}
The number of leaves of a tree can be recovered from its chromatic symmetric function.
\end{cor}

\begin{proof}
The number of leaves is given by the number of connected subgraphs of order $n-1$.
\end{proof}

We now explore some properties of the Frobenius character map that will help us prove a result about forests.  The Frobenius character map \eqref{frob} is not a ring homomorphism.  In fact, for non-disjoint trees $T_1$ and $T_2$ sharing a single vertex, 
$$\mathrm{ch}(\mathcal{K}(T_1))\mathrm{ch}(\mathcal{K}(T_2)) \neq \mathrm{ch}(\mathcal{K}(T_1)\mathcal{K}(T_2)).$$ 
This is because if $T_1$ and $T_2$ share a vertex $j$, every cycle in $\mathcal{K}(T_1)$ that contains the label $j$ increases in length in the product $\mathcal{K}(T_1)\mathcal{K}(T_2)$.  For instance, if $(ab)\in T_1$ and $(bc)\in T_2$, then we have that $4\textrm{ch}((ab))\textrm{ch}((bc)) = p_{(2,2)}$, whereas $4\textrm{ch}((ab)(bc)) = p_{(3)}$. However, for disconnected graphs $T_1$ and $T_2$, $\mathrm{ch}\mathcal{K}(T_1)\textrm{ch}(\mathcal{K}(T_2))= \textrm{ch}(\mathcal{K}(T_1)\mathcal{K}(T_2))$ since the cycles are disjoint. The combination of these two facts proves the following lemma:

\begin{lemma}\label{hommult}
The map $\mathrm{ch}$ has the property
\[\mathrm{ch}(\mathcal{K}(T_1))\mathrm{ch}(\mathcal{K}(T_2))=\mathrm{ch}(\mathcal{K}(T_1)\mathcal{K}(T_2))\] if and only if $T_1$ and $T_2$ are disjoint.  
\end{lemma}

A related result was proved by Tsujie \cite{Tsuj}: a graph has an irreducible chromatic symmetric function if and only if it is connected.  Again, he used entirely different machinery in his proof.

\begin{theorem}\label{irred} \cite{Tsuj}
Let F be a forest. Then F is a tree if and only if the chromatic symmetric function of F is irreducible.
\end{theorem}

\begin{proof}
If $F$ is connected of order $n$ then its CSF has a non-zero coefficient of $p_n$ by Theorem \ref{subtree}. 
Now assume the chromatic symmetric function of $F$ is reducible so that for symmetric functions $\alpha, \beta$,
$$X_F=\alpha \beta.$$ 
Setting the coefficient of $p_n$ equal to $\alpha\neq 0$, the following equation can be obtained
$$(\alpha\beta-(X_F-\alpha p_n))/\alpha= p_n.$$  
This implies that $p_n$ can be written as an algebraic combination of other $p$-symmetric function basis elements, a contradiction.
Due to the fact that the power sum symmetric functions form an algebraically independent basis for $\Lambda$, $p_n$ cannot be expressed as a product of lower-order power series symmetric functions.

In the other direction, assume the forest is disconnected, with connected components $T_1, T_2,...,T_n$.  Then, by Lemma \ref{hommult}, \[\prod_1^n \mathrm{ch}(\mathcal{K}(T_i))=\mathrm{ch}(\prod_1^n \mathcal{K}(T_i)).\]  Therefore the chromatic symmetric function is reducible as $X_{T_1}X_{T_2}=X_{F}$.

\end{proof}

Using Theorem \ref{irred}, Corollary \ref{leafy} can easily be extended to forests by considering the irreducible factors of the CSF of a forest. 
\begin{cor} \cite{Tsuj}
The number of leaves in a forest can be recovered from its chromatic symmetric function. 
\end{cor}

\section{New Tree Developments}
\label{newtreedevelopments}
Inspired by Stanley's conjecture about trees \cite{Stan}, we now investigate the uniqueness of the product $\mathcal{K}_\pi(T)$. The following new results allow for a rephrasing of the original conjecture.

\begin{theorem}\label{exeq}
Let $T_1 = (V_1,E_1), T_2=(V_2,E_2)$ be trees.  Since the order of a graph can be recovered from its CSF, we suppose that $T_1$ and $T_2$ both have vertices labeled $1,...,n$. If there exist orderings $\pi_1$ and $\pi_2$ of the edges $E_1$ and $E_2$ respectively such that $$\mathcal{K}_{\pi_1}(T_1)=\mathcal{K}_{\pi_2}(T_2),$$ then $T_1=T_2$.  Furthermore, if two edges $(ij)$ and $(jk)$ share a common endpoint $j$, and $(ij)$ precedes $(jk)$ in $\pi_1$, then $(ij)$ also precedes $(jk)$ in $\pi_2$.    
\end{theorem}

\begin{proof}
Let $T_1$ and $T_2$ be two trees on $n$ vertices.  Let
 \begin{equation}
 \mathcal{K}_{\pi_1}(T_1)= n! \prod_{\tau_i \in E_1} (1-\tau_i)
 \end{equation}
 and 
\begin{equation}
\mathcal{K}_{\pi_2}(T_2) = n! \prod_{\sigma_i \in E_2} (1-\sigma_i)
\end{equation}
where $\tau_i$ and $\sigma_i$ are the transpositions that are given by the edges of $T_1$ and $T_2$.  

Then, equating $\mathcal{K}_{\pi_1}(T_1) = \mathcal{K}_{\pi_2}(T_2)$, the following equalities arise by comparing cycles of same order:

\begin{align} 1&=1\\
\sum_{\tau_i \in E_1} \tau_i &= \sum_{\sigma_i \in E_2} \sigma_i\\
\sum_{\substack{\tau_i\neq \tau_j \in E_1\\ i<j \textrm{ in }\pi_1}} \tau_i \tau_j &= \sum_{\substack{ \sigma_i, \sigma_j \in E_2  \\ i<j \textrm{ in } \pi_2}} \sigma_i\sigma_j\label{ordering}        
\end{align}

In addition, there are $n-2$ other relationships, one for each cycle length $\leq n$.

 If the following equation holds
 \begin{equation}
     \sum_{\tau_i \in E_1} \tau_i= \sum_{\sigma_i \in E_2} \sigma_i
     \end{equation}
     it is clear that the sets $\{\tau_{i\in I}\}$ and $\{\sigma_{i\in I}\}$ contain the same transpositions, which implies that $T_1 = T_2$.

 We now discuss the information that can be gleaned about the orderings $\pi_1, \pi_2$ from \eqref{ordering}.  Note that $\mathbb{C}[S_n]$ is a non-commutative ring in which the only elements that commute are disjoint cycles.  Suppose that $\mathcal{K}_{\pi_1}(T_1)$ contains the $3$-cycle $\tau_i\tau_j$, where the two transpositions are non-disjoint. Then $\tau_i$ comes before $\tau_j$ in $\mathcal{K}_{\pi_1}(T_1)$. Thus, we have orderings of non-disjoint edges in the products $\pi_1, \pi_2$ that agree.
\end{proof}

Note the converse $\mathcal{K}_{\pi_1}(T_1)\neq \mathcal{K}_{\pi_2}(T_2)$ does not imply that $T_1\neq T_2$.  The path on three vertices, $P_3$, has two orderings of edges that produce different elements of the symmetric group ring: $$(1-(12))(1-(23))\neq (1-(23))(1-(12)).$$ The condition that $\mathcal{K}_{\pi_1}(T_1) = \mathcal{K}_{\pi_2}(T_2)$ is too strong to be practical.

However, the result in Theorem \ref{exeq} can be recast in a more useful form.  Define \[S_G\coloneqq\{\mathcal{K}_{\mathcal{L},\pi}(G)| \mathcal{L}\textrm{ is a labeling of V}, \pi \textrm{ is a permutation of }E\}.\]  By Theorem \ref{exeq}, the following result about $S_G$ also holds.  

\begin{theorem}\label{stronger}
Let $T_1 = (V_1,E_1),T_2= (V_2,E_2)$ be trees.
$S_{T_1}\cap S_{T_2}\neq \emptyset$ if and only if $T_1=T_2$.  
\end{theorem}

The labeling of the vertices in the trees is essential: different labelings $\mathcal{L}_1$ and $\mathcal{L}_2$ of $V_1$ cause $\mathcal{K}_{\mathcal{L}_1, \pi_1}(T_1) \neq \mathcal{K}_{\mathcal{L}_2, \pi_2}(T_1)$ for $\pi_1, \pi_2$ orderings of $E_1$ under the labelings $\mathcal{L}_1, \mathcal{L}_2$ respectively.  For example, if we consider $P_3$ once again, we find that the two labelings of the vertices $1,2,3$ and $2,1,3$ produce disjoint sets of $\mathcal{K}_\pi$ functions.  The first labeling results in \[\{6(1-(12))(1-(23)),6(1-(23))(1-(12))\}\] whereas the second produces \[\{6(1-(12))(1-(13)), 6(1-(13))(1-(12))\}.\]  Both of the elements in the second set contain a $(13)$ term that the elements in the first set do not. In $\mathbb{C}[S_n]$, a relabeling corresponds to conjugation by an element $\sigma \in S_n$.

\begin{prop}
Conjugation of $\mathcal{K}_{\pi}(T)$ for $T= (V,E)$ a tree by a member $\sigma\in S_n$ yields $\mathcal{K}_{\sigma \pi \sigma^{-1}}(T')$ where $T'$ is a relabeling of $T$ and $\sigma \pi \sigma^{-1} = (\sigma e_1\sigma^{-1}, \sigma e_2 \sigma^{-1},...)$ with $e_i \in E$.
\end{prop}

\begin{proof}

Assume 
\begin{equation}
\sigma n!\prod_{(ij)\in E} (1-(ij)) \sigma^{-1} = n!\prod_{(lm)\in E'} (1-(lm)).
\end{equation}

Conjugation represents a graph relabeling because
\begin{equation}
\sigma \prod_{(ij)\in E} (1-(ij))\sigma^{-1}= \prod_{(ij)\in E} \sigma (1-(ij))\sigma^{-1} = \prod_{(ij) \in E} (1-\sigma(ij)\sigma^{-1}).
\end{equation}
\end{proof}
\begin{exmp}
Using $P_3$ once again, consider the labeling in which the vertices are assigned 2---1---3 versus 1---2---3.  Conjugating by $(12)$, the edge $(13)$ is sent to $(12)(13)(12) = (23)$ and the edge $(12)$ is fixed.  The new edges agree with those of the graph under the second labeling. 
\end{exmp}

Therefore, by Theorem \ref{exeq} the equality of $\mathcal{K}_{\pi_1}(T_1)$ and $\mathcal{K}_{\pi_2}(T_2)$ under conjugation implies $T_1$ = $T_2$.

Now, we can make an even stronger statement:

\begin{theorem}
Given two trees $T_1$ and $T_2$, if there exists $\mathcal{K}_{\pi_1}(T_1)\in S_{T_1}$ and $\mathcal{K}_{\pi_2}(T_2)\in S_{T_2}$ along with some element $\sigma \in S_n$ such that \[\sigma\mathcal{K}_{\pi_1}(T_1)\sigma^{-1}=\mathcal{K}_{\pi_2}(T_2)\] then $T_1=T_2$.  
\end{theorem}

Even though this is still insufficient to show Stanley's conjecture that trees are uniquely determined by their CSF's, the following statement, if proved, would imply the conjecture:

\begin{obs}\label{reform}
The chromatic symmetric functions of two trees $X_{T_1}$ and $X_{T_2}$ can be equal if and only if for some  $\sigma \in S_n$ and orderings $\pi_1, \pi_2$
\[\sigma \mathcal{K}_{\pi_1}(T_1)\sigma^{-1}=\mathcal{K}_{\pi_2}(T_2).\]   
\end{obs}

\section{Acknowledgements}

The authors thank Kaiming Zhao for suggesting the proof technique in Theorem 4 and Niven Achenjang for his contributions in the revision process. This work was supported by the Canadian Tri-Council Research Support Fund. The author Ang\`ele Foley was supported by an NSERC Discovery Grant and the authors Joshua Kazdan and Sof\'{i}a Mart\'{i}nez Alberga are supported by the NSF GRFP: Award \# 1650114. This research was conducted at the Fields Institute, Toronto, Canada as part of the 2018 Fields Undergraduate Summer Research Program and was funded by that program.

\end{document}